\documentclass[10pt,twocolumn,a4paper]{IEEEtran}

\makeatletter
\let\NAT@parse\undefined
\makeatother

\newif\ifnotsergio
\newif\ifsergio
\notsergiotrue

\ifsergio
\usepackage[linkcolor={0 0 0}, pdfborder={0 0 0}, linktocpage=true,
hyperindex=true,bookmarks=true, bookmarksnumbered=true]{hyperref}
\fi

\usepackage[english]{babel}
\usepackage[applemac]{inputenc}

\usepackage{flushend}
\usepackage{lscape}
\usepackage{bigints}
\usepackage{graphicx} 
\usepackage[numbers]{natbib}
\usepackage{psfrag}
\usepackage{amsfonts}
\usepackage{amssymb}
\usepackage{amsmath}
\usepackage{graphicx}
\usepackage{psfrag}
\usepackage{subcaption}

\usepackage{ae,epsfig}
\usepackage{verbatim} 
\usepackage{enumerate} 
\usepackage{epstopdf}
\usepackage{tikz}
\usetikzlibrary{shapes,arrows}
\usetikzlibrary{positioning,fit}
\usepackage{color}
\usetikzlibrary{positioning}
\usepackage{tikz}
\usetikzlibrary{shapes,arrows}
\usetikzlibrary{positioning,fit}
\usepackage{bm}
\usetikzlibrary{calc}
\usetikzlibrary{matrix}
\usepackage{algorithmicx}
\usepackage{algorithm}
\usepackage{algpseudocode}

\newcounter{propcount}

\newcounter{remcount}

\newcounter{excount}

\newtheorem{remm}[remcount]{Remark}

\newtheorem{proposition}[propcount]{Proposition}

\newtheorem{ex}[excount]{Example}

\newenvironment{proof}{{\em Proof. }}{\hfill \hspace*{1pt} \hfill $\blacksquare$}

\begin{document}
\author{R. Drummond and S. R. Duncan
\thanks{ R. Drummond and S. R. Duncan are with the Department of Engineering Science, University of Oxford, 17 Parks Road, OX1 3PJ, Oxford, United Kingdom, Email: \{ross.drummond,stephen.duncan\}@ eng.ox.ac.uk. 
}
}

\IEEEoverridecommandlockouts

\title{Accelerated Gradient Methods with Memory}

\maketitle
\thispagestyle{empty}
\pagestyle{empty}

\begin{abstract}
A set of accelerated first order algorithms with memory are proposed for minimising strongly convex functions. The algorithms are differentiated by their use of the iterate history for the gradient step. The increased convergence rate of the proposed algorithms comes at the cost of robustness, a problem that is resolved by a switching controller based upon adaptive restarting. Several numerical examples highlight the benefits of the proposed approach over the fast gradient method. For example, it is shown that these gradient based methods can minimise the Rosenbrock banana function to $7.58 \times 10^{-12}$ in 43 iterations from an initial condition of $(-1,1)$. 

\end{abstract}


\section*{Introduction}

This paper considers the minimisation of continuously differentiable functions $f(x): \mathbb{R}^n \to \mathbb{R}$ belonging to the class $\mathcal{S}_{\mu,L}^n$ of strongly convex and Lipschitz bounded functions, satisfying
\begin{subequations}\begin{align}
\langle f'(x)-f'(y),x-y \rangle \geq \mu \|x-y\|^2, \\
\|f'(x)-f'(y)\| \leq L\|x-y\|,
\end{align}\end{subequations}
with $0 <\mu < L$. For this purpose, black-box gradient methods where the algorithms are only given the parameters $\mu$ and $L$ are developed. The main results of the paper are a set of gradient-based algorithms ($\Sigma_N$ in \eqref{Sigma_N}) that use the iterate memory to accelerate the convergence rate towards the minimiser.  The algorithms are parametrised for quadratic functions such that the error of the mode corresponding to the dominant eigenvalue of the Hessian $w^{(i_{\mu})}_{k}$ converges to zero at the rate
\begin{align}\label{w_rate}
\|w^{(i_{\mu})}_{k}\|_2 = \left(1-\left(\frac{\mu}{L}\right)^{\frac{1}{N}}\right)^k\|w^{(i_{\mu})}_{0}\|_2
\end{align}
for generic $N \in \mathbb{N}$ as described in Section \ref{sec:theta}. This rate is faster than that of the fast gradient method for this mode when $N \geq 3$.


The paper is structured as follows. Section \ref{sec:clas_algs} briefly describes the classical first-order methods of gradient descent and the fast gradient method. The proposed set of algorithms are introduced in Section \ref{sec:algs}, where their parameterisation and robustness for quadratic problems is also discussed. Section \ref{sec:rob} generalises the analysis of these methods to strongly convex problems using absolute stability theory. A switching based control scheme, based upon adaptive restarting, is introduced in Section  \ref{sec:res} to robustify these algorithms and recover their acceleration. The performance of these robustified algorithms is examined in Section \ref{sec:ex_quad} via numerical examples.

\section*{Notation}
The index notation $x^{(i)}_{k} $ indicates the $i^{th}$ element of the vector $x^{(i)}_{k}$ evaluated at time step $k$. Bracketed upper indices are used to relate to elements. $\mathcal{Z}$-transforms of signals will be denoted in capitals when clear from the context.

The fast gradient method is referred to as $\mathcal{FG}$ and gradient descent as $\mathcal{GD}$. In the latter stages of the paper, these algorithms will also be referred to as $\Sigma_1 = \mathcal{GD}$ and $\Sigma_2 = \mathcal{FG}$ so as to be consistent with the notation of the proposed algorithm set  $\Sigma_N$. The reciprocal of the condition number of a function is denoted  $\kappa  = \mu/L$. The identity matrix of dimension $n$ is denoted $I_n$.

\section{First Order Methods: Gradient Descent and the Fast Gradient Method}\label{sec:clas_algs}
The use of first-order methods for minimising strongly convex functions $f(x) \in \mathcal{S}_{\mu,L}^n$ has recently seen revived interest due to their relative computational simplicity, making them ideal for problems with large numbers of decision variables. A typical first order method is initialised with an estimate $x_0 \in \mathbb{R}^n$ of the unique minimiser  of the function $x^* \in \mathbb{R}^n$ and a set of parameters $\theta \in \mathbb{R}^N$ determined from the strong convexity and Lipschitz constants $\mu$ and $L$. The estimate $x_0$ is then updated iteratively by stepping in the direction of the gradient of the function, generating a sequence $x_k \in \mathbb{R}^n$.

The classical first-order method is the gradient descent ($\mathcal{GD}$) algorithm 
\begin{align}\label{gd}
x_{k+1} = x_k-\frac{1}{L} \nabla f(x_k)
\end{align}
which is posed here with a step-length of $1/L$. For functions in $\mathcal{S}_{\mu,L}^n$, this algorithm is guaranteed to converge monotonically towards the optimal value $x^*$ but only at a rate $\propto 1-\mu/L $, which is too slow for many applications. 

A major development in first-order methods was Nesterov's fast gradient method $(\mathcal{FG})$ (or accelerated gradient method) of \cite[Chapter 2]{nest}. Three versions of this algorithm were proposed in \cite[Chapter 2]{nest} and in this paper, the simplest of these will be considered as it recovers the same convergence rate for functions in $\mathcal{S}_{\mu,L}^n$. This method is described by the iterate sequence
\begin{subequations}\label{fg}\begin{align}
x_{k+1} &= y_k-\frac{1}{L}\nabla f(y_k) \\
y_k & = (1+\beta) x_k-\beta x_{k-1}
\end{align}
\end{subequations}
where a step length of $1/L$ has again been used and with the tuning parameter $\beta$ being determined from the reciprocal of the condition number $\kappa = \mu/L$ according to
\begin{align}
\beta = \frac{1-\sqrt{\kappa}}{1+\sqrt{\kappa}}.
\end{align}
At each iteration $k$, the $\mathcal{FG}$ algorithm uses the stored iterate values $x_k$ and $x_{k-1}$ to generate a new point $y_k \in \mathbb{R}^n$ to take the gradient step from. The striking feature of this algorithm is that this rather simple augmentation of gradient descent can result in a dramatic speed up in the convergence rate to $\propto 1-(\mu/L)^{1/2}$ such that
\begin{align}\label{fg_rate}
\|x_k-x^*\|_2^2 \leq (1-(\mu/L)^{1/2})\|x_0-x^*\|^2_2.
\end{align}
Crucially, since the algorithm only uses addition and multiplication, this speed up is achieved without significant sacrifices on computational efficiency. 


Even though the $\mathcal{FG}$ algorithm was developed more than twenty years ago, the impact of the relative simple change of stepping from the point $y_k$ instead of $x_k$ is still being interpreted from a dynamical systems perspective. Using previous iterate information $x_{k-1}$ is said to introduce ``momentum'' or ``inertia'' into the algorithm \cite{momentum}, a term which has been interpreted in terms of both continuous \cite{boyd} and discrete \cite{candes} time second order dynamical systems. It was from this perspective that the symplectic integrator schemes  of \cite{symplectic, jordan1} were developed.

Another interpretation is obtained by considering the algorithms as the feedback interconnection of a linear system with a nonlinear function $u(y_k) = \nabla f(y[k])$ that is both static and sector bounded. Such an approach allows the tools of absolute stability theory \cite[Chapter 6]{khalil} to then be applied. This connection between algorithm design and absolute stability was pioneered in \cite{iqcs} which used the language of integral quadratic constraints (IQCs) to obtain convergence rate bounds that only require solving low-dimensional semi-definite programs (SDP). Lyapunov functions for $\mathcal{FG}$ were proposed in \cite{jordan2, lessard1,lessard3}, with a similar structure to the classical Tsypkin functions of \cite{tsypkin} and \cite{haddad, szego,drum_val}. Exploiting this feedback interpretation, \cite{vanscoy} proposed a parametrisation of $\mathcal{FG}$ with the fastest known convergence rate, with this parametrisation obtained from considering the closed-loop behaviour. Other developments of the fast gradient method include using secant information \cite{heath2}, sum-of-squares programming to determine the algorithm coefficients \cite{sos} and by drawing connections with multi-step methods from numerical integration \cite{intmethods}.

The question still remains as to whether it is possible to radically speed up $\mathcal{FG}$ still further using a similar simple augmentation of the gradient descent algorithm. This is the main purpose of this paper. To this end, a set of algorithms are proposed that use $N \in \mathbb{N}$ historical iterate values to generate the point $y_k$. These algorithms are parameterised by considering quadratic problems such that the iterate corresponding to the dominant eigenvalue of the Hessian converges at the rate $ \propto 1-(\mu/L)^{1/N}$. Note that for $N \geq 3$ this rate is faster than the bound of the fast gradient method, although these algorithms are not robust in the sense discussed in Section \ref{sec:rob1}. To recover performance, an adaptive restarting scheme is proposed and numerical examples showcase the benefits of the approach.



\section{Main Results: Proposed Algorithm set}\label{sec:algs}

For a given $N \in \mathbb{N}$, the minimisation of functions $f(x) \in \mathcal{S}_{\mu,L}^n$ is considered using the $\Sigma_N$ algorithms
\begin{subequations} \label{Sigma_N}
\begin{align} 
x_{k+1} &= y_k-\frac{1}{L}\nabla f(y_k), \label{dyns}
\\
y_k& = \sum_{j = 0}^{N-1} \theta^{(j)} x_{k-j},
\end{align}
whose parameters $\theta  = [\theta_0, \dots, \theta_{N-1}] \in \mathbb{R}^N$ satisfy the affine constraint
\begin{align}
\theta \in \Theta  := \Bigg\lbrace\theta \in \mathbb{R}^N :\sum_{j = 0}^{N-1} \theta^{(j)} & = 1\Bigg\rbrace. \label{constraint}
\end{align}
\end{subequations}
 The constraint \eqref{constraint} on the sum of the parameters allows the error $e_k = x_k - x^*$ to be expressed recursively. The parameter $N$ defines the ``memory'' of the algorithm, with $N = 1$ returning gradient descent $\mathcal{GD}$ and $N = 2$ giving the structure of the fast gradient method $\mathcal{FG}$. This section is concerned with algorithms where $N \geq 3$. The notation $\Sigma_N$ denotes an algorithm with memory $N$, with $\Sigma_1 = \mathcal{GD}$ and $\Sigma_2 = \mathcal{FG}$. The results of this section build upon the ideas presented in \cite{drum} for minimising the quadratic cost functions of model predictive control.

The algorithms $\Sigma_N$ \eqref{Sigma_N} have a similar structure to $\mathcal{FG}$, as they perform a gradient step from a point $y_k$ generated from $N$  past iterates. The difference between them comes from the increased number of past iterates used in $\Sigma_N$. It is not required that the algorithm parameters $\theta \in \Theta$ be positive, so the generated point $y_k$ may not lie within the convex hull of the previous iterates. This fact is used in Section \ref{sec:non_convex} for the minimisation of non-convex functions.

\subsection{Setting the parameters $\theta$}\label{sec:theta}
The  $\Sigma_N$ algorithm parameters $\theta$ still need to be defined. These parameters are set by considering the minimisation of quadratic functions $f_Q(x):\mathbb{R}^n \to \mathbb{R}$
\begin{align}
f_Q(x) = \frac{1}{2}x^THx + h^Tx
\end{align}
belonging to the class $f_Q \in \mathcal{Q}^n_{\mu,L} \subset \mathcal{S}_{\mu,L}^n$ where the eigenvalues $\lambda_i$  of the Hessian $H$ are restricted to be positive and lie within the range $\lambda_i \in [\mu,L]~\forall i =  1, \dots, n$.

Considering the iterate update of $\Sigma_N$ in \eqref{dyns} with $f(\cdot) = f_Q(\cdot)$, then by subtracting the minimiser $x^* \in \mathbb{R}^n$ from both sides and adding $\frac{1}{L}\nabla f(x^*) = 0$ to the RHS generates the recursive error sequence
\begin{align}
e_{k+1} &= \sum_{j = 0}^{N-1} \theta^{(j)} \left(1-\frac{1}{L}H\right)e_{k-j}. \label{error_sys}
\end{align} 
 Introducing the function $m(s): [\kappa,1] \to [0, 1-\kappa]$,
\begin{align}
m(s) = 1-s,
\end{align}
then this error system \eqref{error_sys} admits an eigen-decomposition
\begin{align} \label{modes_i}
w^{(i)}_{k+1} &= \sum_{j = 0}^{N-1} \theta^{(j)} m(\lambda_i/L)w^{(i)}_{k-j}, \quad i = 1, \dots, n,
\end{align} 
with $\lambda_i$ being the $i^{th}$ eigenvalue of the Hessian $H$. 

The parameters $\theta$ are set according to the mode associated to the dominant eigenvalue of the Hessian, namely $ \lambda_{i_{\mu}} = \mu$. This mode evolves according to 
\begin{align}\label{w_i_mu}
w^{(i_{\mu})}_{k+1} &= \sum_{j = 0}^{N-1} \theta^{(j)} m(\kappa)w^{(i_{\mu})}_{k-j},
\end{align} 
and has characteristic equation
\begin{align}\label{char}
p_N(r;m(\kappa)) = r^N - \sum_{j = 0}^{N-1} \theta^{(j)} m(\kappa)r^{N-1-j}= 0.
\end{align} 
Denoting the roots of this polynomial as $r_1, \, r_2, \,\dots, \,  r_N$, then the evolution of the state of the above linear system can be written
\begin{align}
w^{i_{\mu}}_{k+1} & \propto c_1r_1^k + c_2r_2^k + \dots + c_Nr_N^k,
\end{align} 
with the roots $r_j$ determining the rate of convergence. The problem is then to choose the parameters $\theta \in \mathbb{R}^N$ (subject to the constraint \eqref{constraint}) that minimises the moduli of the roots $r_j$ for  $j = 1, \dots N$. 

Defining the root radius of a polynomial $p_N(r)$ 
\begin{align}
\rho(p_N) = \max \{|r| \, : \, p_N(r;m(\kappa)) = 0, r \in \mathbb{C}\},
\end{align}
with the polynomials $p_N(r)$ contained within the set
\begin{align}
P = \Bigg\lbrace r^N - \sum_{j = 0}^{N-1} \theta^{(j)} m(\kappa)r^{N-1-j} : \sum_{j = 0}^{N-1} \theta^{(j)} = 1, \theta^{(j)} \in \mathbb{R}\Bigg\rbrace,
\end{align}
then this problem can be cast as that of finding $\theta \in \Theta$ that transforms  $p_N$ into the optimal polynomial $p^*_N$ minimising
\begin{align} \label{root_opt}
\rho^* := \rho(p_N^*)=  \inf_{p_N \in P} \quad \rho(p_N).
\end{align}

In general, the problem of globally minimising the root radius of a polynomial is known to be both non-convex and not Lipschitz. However, in \cite{roots}, it was shown that there exists an analytic solution to this problem if the coefficients of the polynomials $p_N(r;m(\kappa))$ are subject to a single affine constraint, which is the class of polynomials considered here with the affine constraint given by \eqref{constraint}. 

Theorem 6 from \cite{roots} states that if the parameters $\theta$ satisfy the affine constraint \eqref{constraint} and moreover are allowed to be complex  $\theta \in \mathbb{C}^N$, then the solution $p_N^*(r)$ to the root radius optimisation \eqref{root_opt} is the polynomial with a common real root $\gamma \in \mathbb{R}$
\begin{align}
p_N^*(r) = (r-\gamma)^N.
\end{align}
Similarly, if the parameters satisfy the affine constraint \eqref{constraint} but are restricted to be real $\theta \in \mathbb{R}^N$, then Theorem 1 of \cite{roots} states that the optimal polynomial is instead
\begin{align}
p_N^*(r) = (r+\gamma)^{N-j}(r-\gamma)^{j},
\end{align}
for some integer $j$ with $0 \leq j \leq N$. 

These theorems form the basis for the parameter choice adopted here; choosing $\theta$ such that $p_N(r;m(\kappa))$ has a common root $\gamma$. Even though this choice may not be optimal according to Theorem 1 of \cite{roots}, it gives an analytic solution to the problem which is useful for determining the algorithm speed-up.
\begin{proposition}[Choosing $\theta$] \label{prop:roots}
If such a choice exists, then the parameters $\theta \in \mathbb{R}^N$ should be chosen such that the polynomial $p_N(r;m(\kappa))$ has a common root at $\gamma \in \mathbb{R}$. This common root is given by 
\begin{align} \label{root_N}
\gamma = 1-\left(\frac{\mu}{L}\right)^{\frac{1}{N}}.
\end{align}

\end{proposition}

\begin{proof}



To compute the common root $\gamma$ in \eqref{root_N}, the parameters $\theta$ of the polynomials $p_N(r;m(\kappa))$ are related to its roots $r_1,r_2, \dots, r_N$ using Vi\`{e}te's formulas \cite{vinberg}. These formulas are obtained from the elementary symmetric polynomials and state that for any polynomial such as $p_N(r;m(\kappa))$ defined by real or complex coefficients $\theta$, then the following equations are satisfied
\begin{subequations}\begin{align*}
r_1 + r_2 +r_3+ \dots + r_N &= \theta_{N-1}m(\kappa),
\\ 
r_1 r_2+ r_1r_3 + \dots r_1r_N +  \dots + r_{N-1}r_N &= -\theta_{N-2}m(\kappa),
\\ 
r_1 r_2r_3+ \dots  + r_1r_2r_N +  \dots + r_{N-2}r_{N-1}r_N  & = \theta_{N-3}m(\kappa),
\\
...........................................................
\\
r_1r_2r_3 \dots r_N = (-1)^{N+1}&\theta_0m(\kappa).
\end{align*}\end{subequations}
If $\theta$ is chosen according to Proposition \ref{prop:roots},  then the polynomial $p_N(r;m(\kappa))$ has a common root $\gamma$, as in $r_1 = r_2 = \dots = \gamma$, and the Vi\`{e}te's formulas collapse to
\begin{subequations}\begin{align*}
N\gamma & = \theta_{N-1}m(\kappa),
\\ 
\begin{pmatrix} N \\ 2 \end{pmatrix}\gamma^2 & = -\theta_{N-2}m(\kappa),
\\ 
\begin{pmatrix} N \\ 3 \end{pmatrix}\gamma^3 &= \theta_{N-3}m(\kappa),
\\
................&........................
\\
\gamma^N &= (-1)^{N+1}\theta_0m(\kappa).
\end{align*}\end{subequations}
The constraint $\sum_{j = 0}^{N-1} \theta^{(j)} = 1$ means that after summing these formulas,  the common root satisfies the polynomial
\begin{align}
\gamma^N+\sum^{N-1}_{j = 1} (-1)^{j+1}\begin{pmatrix} N \\ j \end{pmatrix}\gamma^j = (-1)^{N+1}m(\kappa).
\end{align}
Alternatively, this polynomial can be written 
\begin{align}
(\gamma-1)^N = (-1)^N\kappa,
\end{align}
giving \eqref{root_N}.
\end{proof}

Note the similarity between this equal root condition and the critical damping analysis of the fast gradient method in \cite{candes}. In fact, with this condition, the fast gradient method is recovered with $N = 2$.



\subsection{Robustness}\label{sec:rob1}

Choosing $\theta$ according to Proposition \ref{prop:roots} means that the iterate update \eqref{w_i_mu} associated to the dominant eigenvalue of the Hessian converges as in \eqref{w_rate}. However, this says nothing about the convergence of the other modes in \eqref{modes_i}. The characteristic equation for these modes is
\begin{align}
p_N(r;m(\lambda_i/L)) = r^N - \sum_{j = 0}^{N-1} \theta^{(j)} m(\lambda_i/L)r^{N-1-j}= 0
\end{align}
which differs from the polynomial $p_N(r;m(\kappa))$ in \eqref{char} by perturbations in the polynomial coefficients $m(\cdot)$. Such perturbations can significantly influence the root locations and even cause the polynomial to no longer be Schur \cite{bdo}.  

This results in a robustness issue which is illustrated in Figure \ref{fig:m}, where $\rho(p_N(r;m(\lambda_i/L))) $ is plotted against $m(\lambda_i/L)$ for a problem with $\kappa = 0.01$. At the point $m(\lambda_i/L) = m(\kappa)$, increasing $N$ pushes the root radius down according to \eqref{root_N} but increases it elsewhere. There is even a range of $m(\lambda_i/L)$ for which $\Sigma_5$ has a root radius greater than 1, inducing divergence. These algorithms can then be said to be fragile, calling for the switching controller of Section \ref{sec:res} to compensate for this lack of robustness and to recover performance. 


\begin{figure}
\centering
\includegraphics[width = 1\columnwidth]{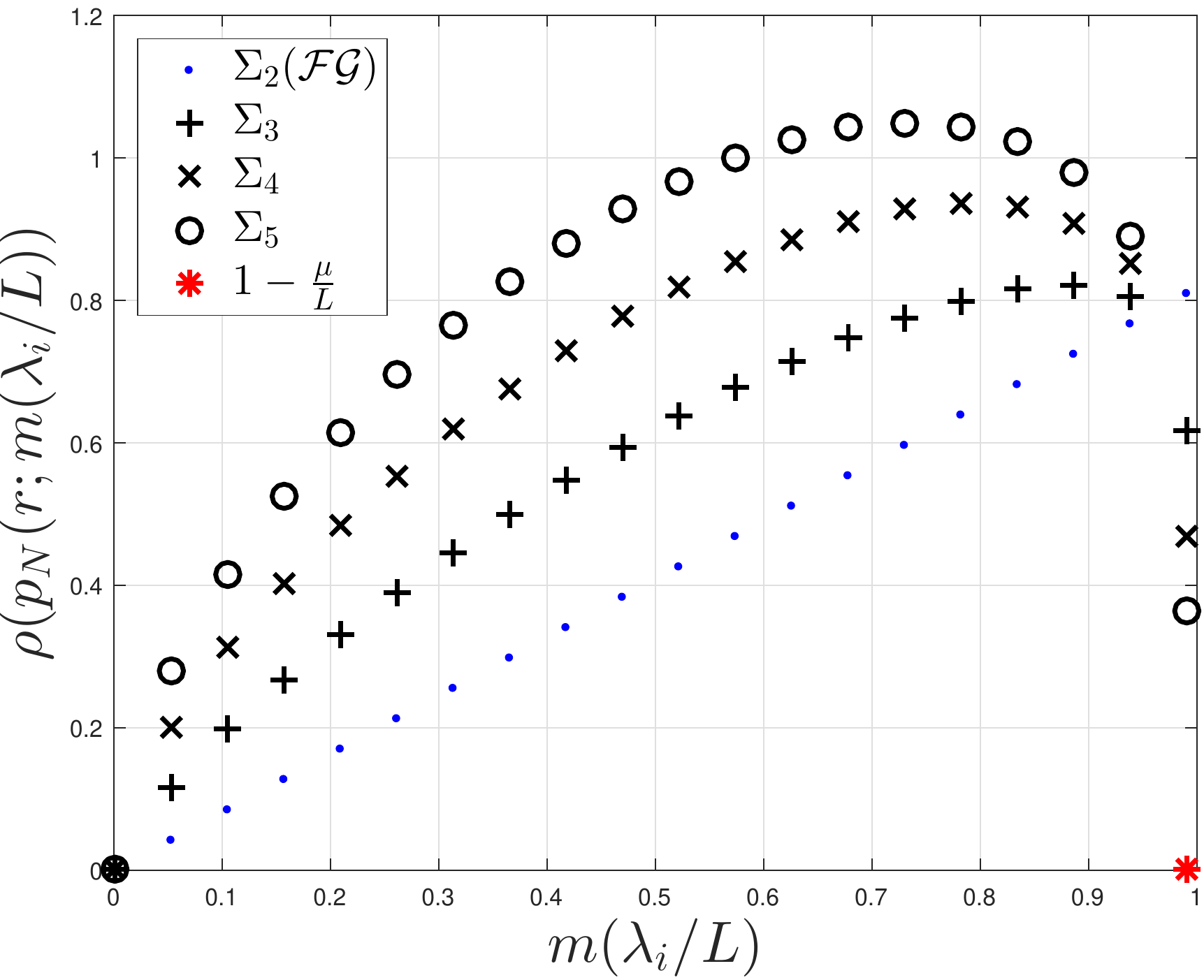}
\caption{Variations of the root radius of $p_N(r;m(\lambda_i/L))$ as a function of $m(\lambda_i/L)$ when $\kappa = 0.01$. The root radius at $m(\lambda_i/L) = m(\kappa)$ dramatically drops according to \eqref{root_N} but increases elsewhere and in fact even exceeds one.  This results in a divergence  which is controlled using a switching controller.}
\label{fig:m}
\end{figure}

\section{Analysis of Strongly Convex Functions}\label{sec:rob}

The algorithm design and robustness analysis of the previous section was applied to quadratic problems as it meant that the algorithm dynamics were linear, making the design problem  tractable. This section discusses the generalisation of this analysis to generic functions $f(x) \in \mathcal{S}_{\mu,L}^n$ which lead to nonlinear algorithm dynamics. 

To this end, the absolute stability approach of \cite{iqcs} is adopted. In this setting, the gradient of the cost function is regarded as a nonlinear function $u(x):\mathbb{R}^n \to \mathbb{R}^n$ that is static
and sector bounded, satisfying the quadratic inequality
\begin{align}
\begin{bmatrix}y_k-y^* \\ u_k-u^*\end{bmatrix}
\begin{bmatrix}-2mLI_n & (L+m)I_n \\(L+m)I_n & -2I_n\end{bmatrix}
\begin{bmatrix}y_k-y^* \\ u_k-u^*\end{bmatrix} \geq 0.
\end{align}
If $\nabla f(x) \in S^n_{\mu,L}$, then $u(x) = \nabla f(x)$ satisfies these criteria \cite{iqcs}. One can apply a loop transformation \cite{khalil} on the nonlinearity so as to normalise the sector from $[\mu,L]$ to $[0,1]$ by writing the gradient as 
\begin{align}
\nabla f(x) = (L-\mu)\left(\frac{\nabla f(x)- \mu x}{L-\mu}+ \frac{\mu x}{L-\mu}\right),
\end{align}
or
\begin{align}
\nabla f(x) = (L-\mu)\nabla{f}_{([0,1])}(x)+ \mu x
\end{align}
with the transformed nonlinearity $u_{[0,1]}(x) = \nabla f_{[0,1]}(x) =\frac{\nabla f(x)- \mu x}{L-\mu}$ now lying within the sector $[0,1]$.

This loop transformation allows the algorithms $\Sigma_N$ to be expressed as
\begin{align}
x_{k+1} = \sum_{j = 0}^{N-1}m(\kappa)\theta^{(j)} x_{k-j} -\frac{L-\mu}{L}\nabla{f}_{[0,1]}(y_k).
\end{align}
whose $\mathcal{Z}$-transform, with $u_{[0,1]}(y_k) = \nabla {f}_{[0,1]}(y_k)$, is
\begin{align}
X(z)&=  -\left(\frac{L-\mu}{L}\right)\left(zI_n-\sum_{j = 0}^{N-1}m(\kappa)\theta^{(j)} z^{-j}I_n\right)^{-1}U_{[0,1]}(z),\\
Y(z)& = \left(\sum_{j = 0}^{N-1}m(\kappa)\theta^{(j)} z^{-j}I_n \right)X(z).
\end{align}
By using dimension reduction \cite{iqcs}, the transfer function of the system of interest is then reduced to
\begin{align}
&\frac{Y(z)}{U_{[0,1]}(z)}  = G_N(z) = -\left(\frac{L-\mu}{L}\right)\frac{\sum^{N-1}_{j = 0} m(\kappa)\theta^{(j)} z^{-j}}{ z +  \sum^{N-1}_{j = 0} m(\kappa)\theta^{(j)} z^{-j}}. \nonumber
\end{align}
This transfer function allows the algorithm to be expressed as the feedback interconnection of the linear system $G_N(z)$ with the nonlinearity $u_{[0,1]}(y_k)=  \nabla f_{[0,1]}(y_k)$ as illustrated in Figure \ref{fig:Lurie1}. 

The stability of such feedback systems can be verified using standard techniques such as constructing Lyapunov functions (e.g. those proposed by Tsypkin \cite{tsypkin} and Szego \cite{szego}) and searching for Zames-Falb multipliers. These results give a framework by which one can compute many system properties, such as reachable sets, stability margins, exponential convergence rates and local stability regions \cite{boyd} for instance. For the sake of brevity, the details of these methods are not included here.  It is stressed that this analysis is applicable for all functions satisfying the sector conditions, generalising the analysis to the whole class of functions $f(x) \in \mathcal{S}_{\mu,L}^n$.


The design of Proposition \ref{prop:roots} can then be considered from the open loop perspective, as the parameters $\theta$ were set such that the poles of $G_N(z)$ were fixed to $\gamma$. Even though these poles are Schur, and hence $G_N(z)$ is stable, this does not guarantee the stability of the overall feedback loop. In fact, it is known that introducing feedback can in fact lead to instability, as reflected in the analysis of Section \ref{sec:rob1}.  

A more nuanced design criteria would then be posed in the $z$-domain and take into consideration the passivity of the feedback structure more explicitly, in a manner similar to \cite{vanscoy}. This approach was not adopted here as analytic results were desired. 

 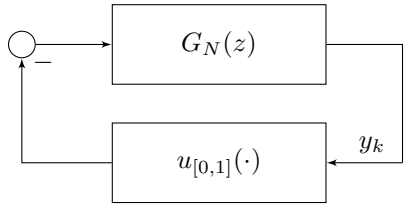
\begin{figure}
\tikzstyle{block} = [draw, fill=white!20, rectangle, 
    minimum height=3em, minimum width=8em]
\tikzstyle{sum} = [coordinate]
\tikzstyle{input} = [coordinate]
\tikzstyle{output} = [coordinate]
\tikzstyle{pinstyle} = [pin edge={to-,thin,black}]
\tikzstyle{sum2} = [draw, fill=white!8, circle, 
    minimum height=1em, minimum width=1em]
\center
\begin{tikzpicture}[auto, node distance=2.0 cm,>=latex']


    \node [block] (G) {$G_N(z)$};
\node [block, below =0.5cm of G] (phi) {$u_{[0,1]}(\cdot)$};
\node [sum, right =1cm of phi] (pinR) {};
\node [sum2, left =1cm of G] (pinL) {};
    
    \draw [-] (G) -| node {} (pinR);
  \draw [->] (phi) -|(pinL) ;
 \draw [->] (pinR) --(phi.east)node[pos=0.4,above] {$y_k$};
  \draw [->] (pinL) --node[pos=0.1,below] {$-$}(G);
\end{tikzpicture}
\caption{First order optimisation methods represented as the feedback interconnection of a linear system with a sector bounded nonlinearity corresponding to the gradient of the function $u_{[0,1]}(x) = \nabla f_{[0,1]}(x)$. }
\label{fig:Lurie1}
\end{figure}

\section{Control by Restarting}\label{sec:res}
To stabilise and recover the performance of $\Sigma_N$ , the use of adaptive restarting \cite{candes} as a form of switching control is proposed. The typical notion of restarting an algorithm means to perform a gradient descent step at a given iteration to promote monotonic convergence and was initially applied after a pre-defined iteration number \cite{nesterovrestart}. In \cite{candes}, an improved adaptive restarting approach was proposed where the gradient descent step was only applied if the projected iterates caused a certain condition to trigger, with a typical trigger being that the projected iterate would increase the value of the function as in $f(x_{k+1}) \geq f(x_k)$. This is the trigger condition adopted here although others exist, such as the gradient  condition of \cite{candes}. When such a condition is triggered then the algorithm switches to $\mathcal{GD}$, guaranteeing a decrease in the function value, before switching back to the non-monotonic algorithm (such as $\mathcal{FG}$) to take advantage of its faster convergence rate. Here, this notion is generalised to switch between the many algorithms in the set $\Omega  = \{\Sigma_j \, :\, j = 1, \dots, N\}$, not just between $\mathcal{GD}$ and $\mathcal{FG}$.

Denote $\mathcal{T}_N(x_{k}, \, \dots, \, x_{k+1-N})$ as the iterate update operator for $\Sigma_N$, such that \eqref{dyns} can be written more compactly as
\begin{align}
x_{k+1} = \mathcal{T}_N(x_{k}, \, \dots, \, x_{k+1-N}).
\end{align}
With this notation the implementation of the fast gradient method with adaptive restarting can be expressed as Algorithm 1 (where a function based restart is used).
\begin{algorithm}\label{alg:fg_res}
\caption{$\mathcal{FG}$ with Adaptive Restarting}
\begin{algorithmic}
\Require{$x_0 \in \mathbb{R}^n$, $x_{-1} = x_0$, $L$ and $\beta$.}
\For{$k = 0,1, 2, \dots$} 
 \State $\hat{x}_{k+1} = \mathcal{T}_2(x_k,x_{k-1})$ 
\If {$f(\hat{x}_{k+1})-f(x_k) >0$}
    \State $x_{k+1} = \mathcal{T}_1(x_k)$ 
\Else
\State $x_{k+1} = \hat{x}_{k+1}$
\EndIf
\EndFor
\end{algorithmic}
\end{algorithm}

This adaptive approach is generalised in Algorithm 2 to allow the algorithms to switch between the whole set $\Omega$.

\begin{algorithm}
\caption{$\Sigma_N$ with Adaptive Restarting ($\Sigma^{Re}_N$)}
\begin{algorithmic}
\Require $N \in \mathbb{N}$, $x_{1-N:0} \in \mathbb{R}^{N n }$, $L$ and $\theta$.
\For {$k = 0,1, 2, \dots$} 
 \State $\hat{x}_{k+1} = \mathcal{T}_N(x_k,\,x_{k-1}, \, \dots\, x_{k+1-N})$ 
\If {$f(\hat{x}_{k+1})-f(x_k) >0$}
    \State $\hat{x}_{k+1} = \mathcal{T}_{N-1}(x_k, \,x_{k-1},\, \dots, \, x_{k+2-N})$ 
    \If {$f(\hat{x}_{k+1})-f(x_k) >0$}
        \State $\hat{x}_{k+1} = \mathcal{T}_{N-2}(x_k, \,x_{k-1}, \, \dots , \, \, x_{k+3-N})$ \\
        \State $\dots$ \\
            \If {$f(\hat{x}_{k+1})-f(x_k) >0$}
        \State ${x}_{k+1} = \mathcal{T}_{1}(x_k)$
        \EndIf \\
        \quad \quad \quad \quad \quad $\dots$ \\
    \EndIf
\Else
\State $x_{k+1} = \hat{x}_{k+1}$
\EndIf
\EndFor
\end{algorithmic}
\end{algorithm}

A second scheme that also guarantees monotonic convergence is to update each algorithm in the set $\Omega$ simultaneously and then to select the iterate that minimises the cost function the most. This is illustrated in the multi-legged algorithm $\Sigma_N^{ml}$ of Algorithm 3, so named because the algorithm tentatively steps from the point $x_k$ in various directions. Monotonic convergence is guaranteed with this scheme for $f(x) \in \mathcal{S}_{\mu,L}^n$ by including the $\mathcal{GD} = \Sigma_1$ algorithm in $\Omega$. 

\begin{algorithm}
\caption{Multi-Legged Algorithm ($\Sigma_N^{ml}$)}
\begin{algorithmic}
\Require $N \in \mathbb{N}$, $x_{1-N:0} \in \mathbb{R}^{N n }$, $L$ and $\theta $.
\For {$k = 0,1, 2, \dots$} 
 \State $\hat{x}_{k+1,1} = \mathcal{T}_1(x_k)$ 
 \State $\hat{x}_{k+1,2} = \mathcal{T}_2(x_k,x_{k-1})$ 
 \State $\dots$
 \State $\hat{x}_{k+1,N} = \mathcal{T}_N(x_k,x_{k-1}, \dots, x_{k+1-N})$ 
 \State Define  $\mathbb{X}_{k+1} = \{\hat{x}_{k+1,1}, \, \dots, \, \hat{x}_{k+1,N}\}$  
	 \State $x_{k+1} = \min_{x \in \mathbb{X}_{k+1}}~f(x)$.
\EndFor
\end{algorithmic}
\end{algorithm}

Controlling the algorithms by adaptive restarting in this way allows algorithms for which the operator $\mathcal{T}_{N}$ is not be globally contracting to be considered. The use of such operators was found to speed up convergence (as discussed in the examples of Ssection \ref{sec:ex_quad}) and suggests connections to the fundamental limitations between controller performance and robustness.

The need to evaluate the function at each iteration increases the computational cost of running these algorithms, and this could be reduced by using the gradient based trigger condition. Memory requirements are also increased with these methods. But, as demonstrated in the following examples section, adaptive restarting brings both added performance and robustness over the fast gradient method, with robustness being highlighted as a necessary feature for the use of these algorithms in practise \cite{robust_grad}. It is also noted that these approaches are highly parrelisable which should recover some computational speed in an efficient implementation.


 \begin{figure*}
\centering
\begin{subfigure}[b]{0.45\textwidth}\centering
\includegraphics[width=0.95\textwidth]{{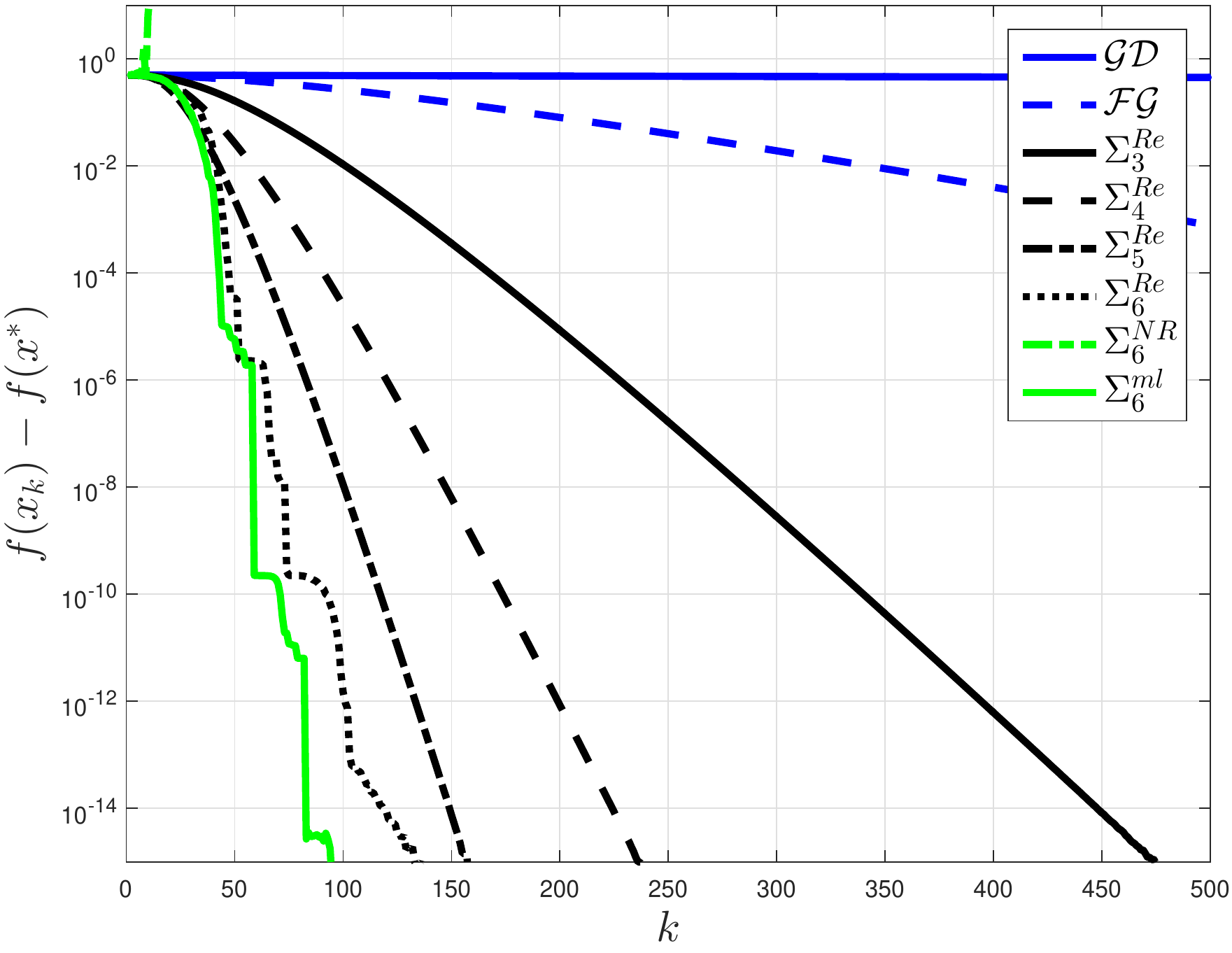}}
\caption{Function iterates of $f_{ex:1}(x)$.}
\label{fig:1f}
\end{subfigure}
\begin{subfigure}[b]{0.45\textwidth}\centering
\includegraphics[width=0.95\textwidth]{{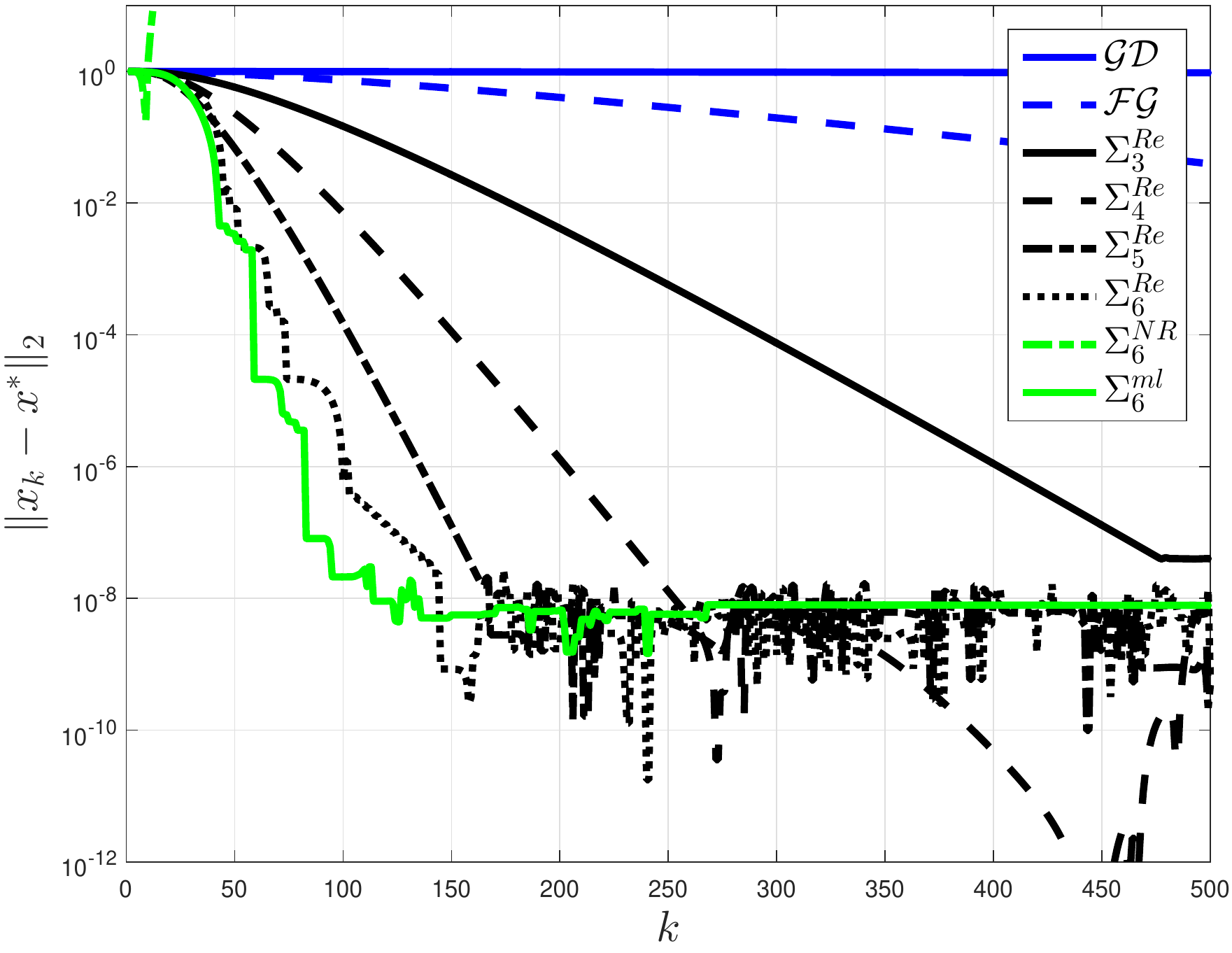}}
\caption{Error trace of $f_{ex:1}(x)$.}
\label{fig:1e}
\end{subfigure}
\label{fig:ex1}
\\
\begin{subfigure}[b]{0.45\textwidth}\centering
\includegraphics[width=0.95\textwidth]{{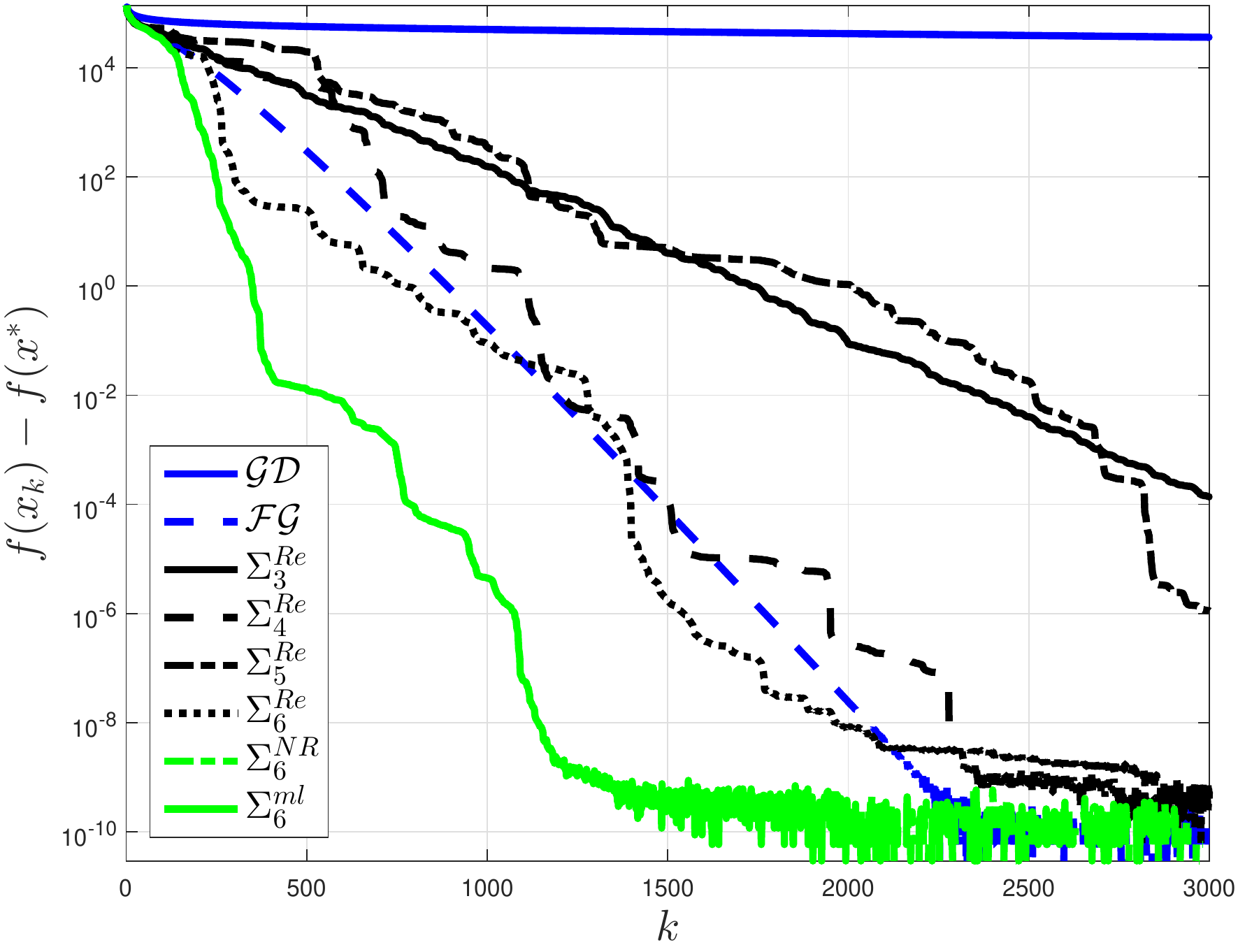}}
\caption{Function trace of $f_{ex:2}(x)$.}
\label{fig:2f}
\end{subfigure}
\begin{subfigure}[b]{0.45\textwidth}\centering
\includegraphics[width=0.95\textwidth]{{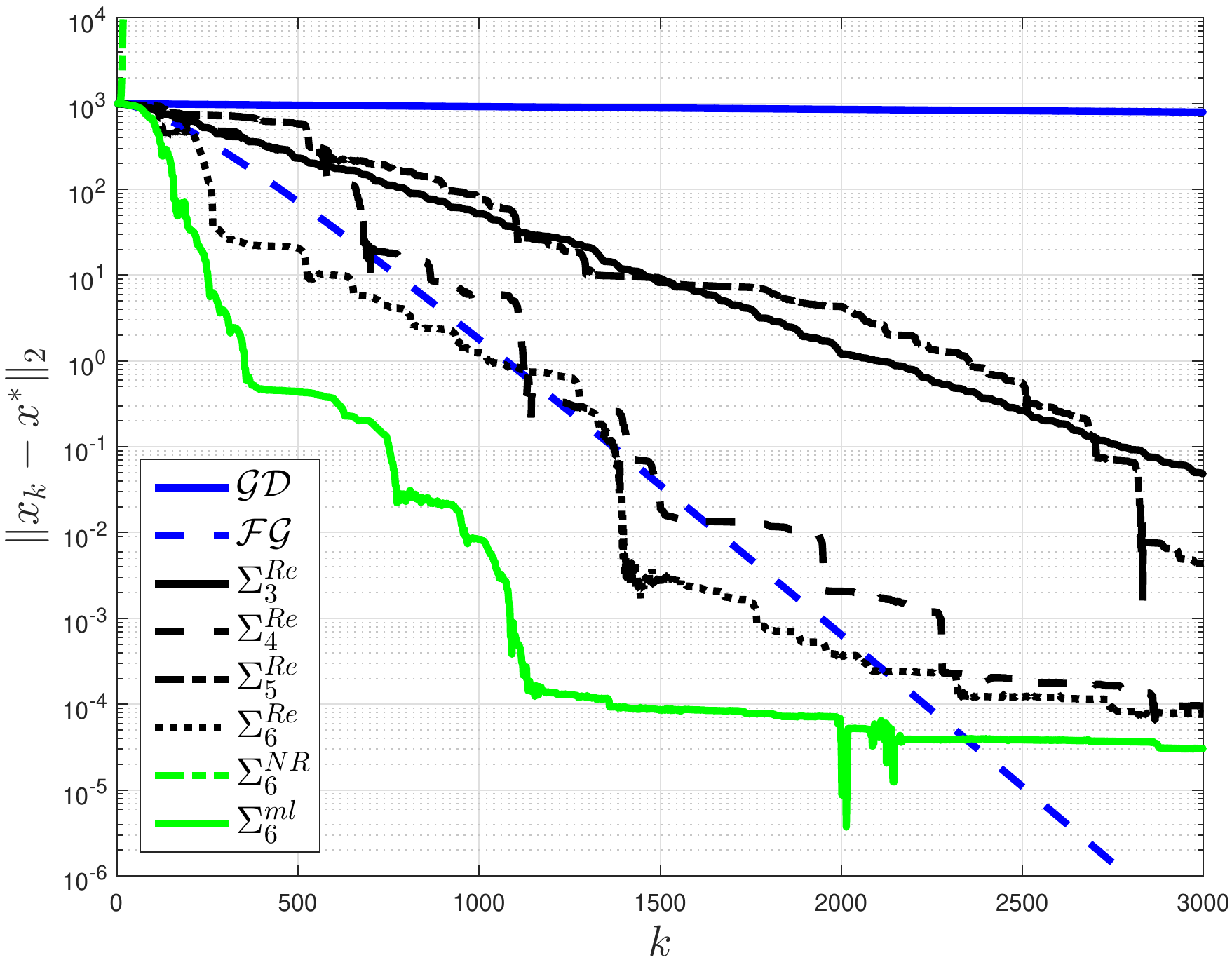}}
\caption{Error trace of $f_{ex:2}(x)$.}
\label{fig:2e}
\end{subfigure}
\caption{Minimisation of the functions $f_{ex:1}(x)$ and $f_{ex:2}(x)$ using gradient descent $\mathcal{GD}$, the fast gradient method $\mathcal{FG}$, the proposed algorithms with adaptive restarting $\Sigma^{Re}_N$ (Algorithm 2) for various $N$ and the multi-legged algorithm $\Sigma_N^{ml}$ for $N = 6$ (Algorithm 3). The multi-legged algorithm $\Sigma_6^{ml}$ converged quickest until the function values became small.}
\label{fig:ex2}
\end{figure*}

\section{Examples}\label{sec:ex_quad}
This section evaluates the performance of the proposed first-order algorithms via several numerical examples.

\subsection{Quadratic problems}
\subsubsection{Hessian eigenvalues clustered at $L$}
We begin with the follow quadratic functions whose eigenvalues are clustered at the Lipschitz constant $L$
\begin{align}\label{f_ex1}
f_{ex:1}(x) = (x^{(1)})^2 +\bm{1} x+ \sum_{j = 0}^{n-2} (L-j) (x^{(j+2)})^2.
\end{align}
The minimisation of this function with $\mu = 1$, $L = 10^4$ and $n = 10^3$ from an initial condition of $x_0 = \bm{0}$ using $\Sigma_{1}^{Re}, \, \Sigma_{2}^{Re}, \dots, \, \Sigma_{6}^{Re}$ with adaptive restart and the multi-legged algorithm $\Sigma_N^{ml}$ is shown in Figures \ref{fig:1e} and \ref{fig:1f}. As a reminder, $\Sigma_1^{Re}$ corresponds to gradient descent, $\Sigma_6^{ml}$ is the multi-legged algorithm for $N = 6$ whilst $\Sigma_2^{Re}$ and $\Sigma_{3:N}^{Re}$ are respectively the fast gradient method and the proposed algorithms with adaptive restart. Also plotted on this figure is the trace of the algorithm $\Sigma_6$ without adaptive restarting, denoted $\Sigma_6^{NR}$.   

Increasing the memory $N$ led to faster convergence speed, with $\Sigma_6^{Re}$ being an order of magnitude faster than the fast gradient method. This was until the function values became small from which the function based adaptive restarting scheme struggled. Without restarting, the iterates of $\Sigma_6$ were actually divergent for this problem, highlighting the need for stabilisation by restarting. This simple example shows how the fast gradient method may  not be ``optimal'' for many strongly convex problems and initial conditions. 

\subsubsection{A more even distribution of Hessian eigenvalues}

In contrast, the following quadratic function \eqref{f_ex1}
\begin{align}
f_{ex:2}(x) = \frac{1}{2}x^T\begin{bmatrix}1 & 1 & 1  & \dots  & 1 \\
1 & 2 & 1 & \dots & 1 \\
1 & 1 & 3 & \dots & \vdots\\
\vdots  &  \vdots & \vdots &\ddots & 1\\
1 & 1 & \dots & 1 &  n \end{bmatrix} x
+ \begin{bmatrix}1 \\ 2 \\3 \\ \vdots \\  n \end{bmatrix}^Tx.
\end{align}
has a more even distribution of eigenvalues. The minimisation of this function with $n = 10^3$ is shown in Figures \ref{fig:2e} and \ref{fig:2f}, with the fastest convergent rate once again being observed with $\Sigma_6^{ml}$. The condition number of this function is $ 1.37\times 10^4$, but because the eigenvalues are spread more evenly in the range $[\mu,L]$, the performance improvement of the algorithms $\Sigma_N^{Re}$ were not as substantial as for $f_{ex:1}(x)$. However, $\Sigma_6^{ml}$ still obtains a significantly faster convergence rate than $\mathcal{FG}$.

\begin{figure}
\includegraphics[width=0.5\textwidth]{{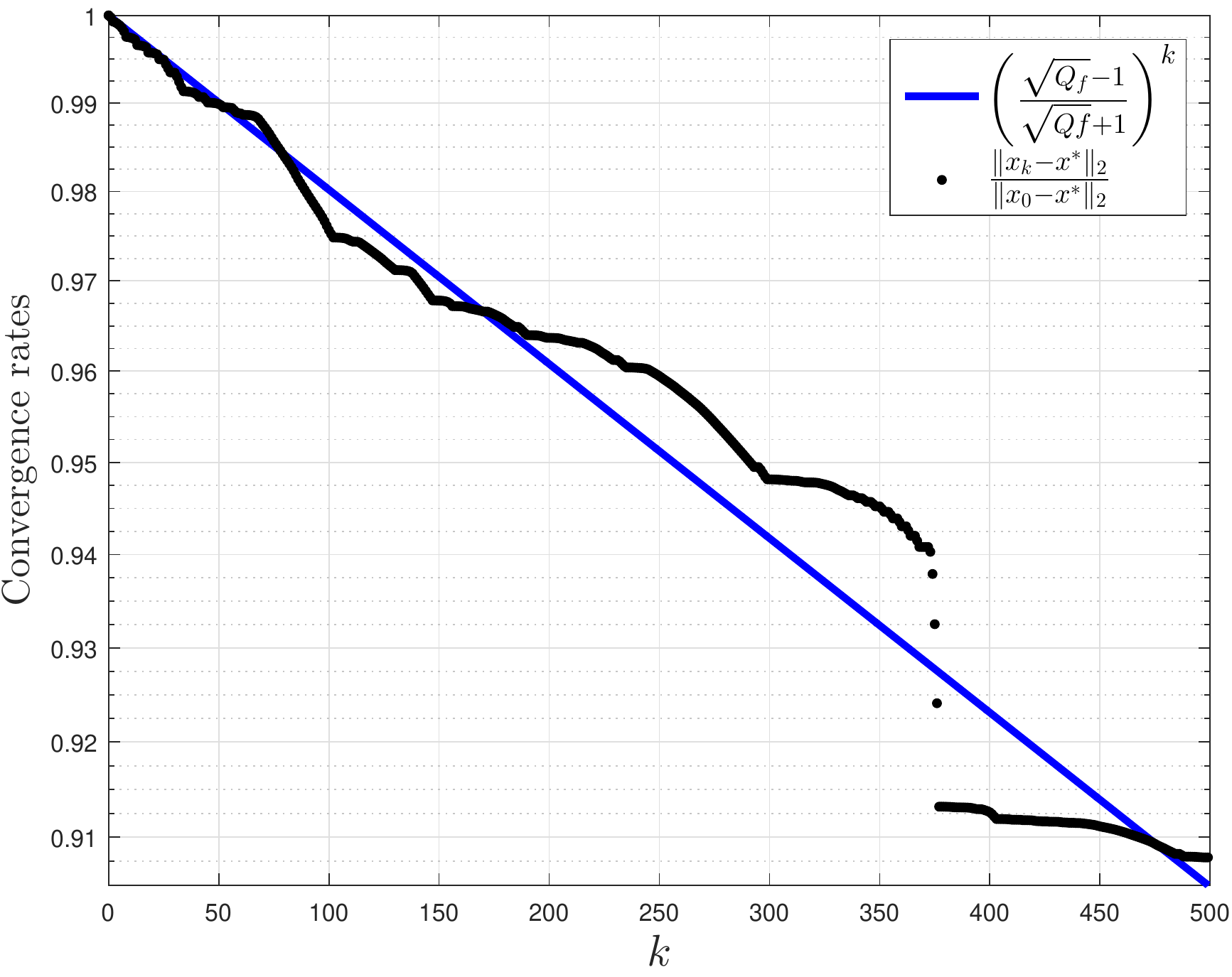}}
\caption{Convergence of the truncated version of the Nesterov's counter-example $f_{Nest}^n$ using $\Sigma_6^{ml}$. The truncation allows the generated iterates to exceed the bound given in \eqref{opt_bound}. }
\label{fig:nest_ce}
\end{figure}
\subsection{Nesterov's counter-example}\label{sec:opt}
Any proposed first-order method that compares itself against the fast gradient method has to take into consideration that algorithm's optimality. This label is discussed here.

The fast gradient method is referred to as an ``optimal'' algorithm \cite[Chapter 2]{nest}, with the justification being that there exists a function
\begin{align}\label{worst_case}
&f_{Nest}(x) = \frac{1}{2}\mu\|x\|^2 \\
& + \frac{\mu(Q_f-1)}{4}\left\lbrace\frac{1}{2}\left[(x^{(1)})^2 + \sum_{i = 1}^{\infty} (x^{(i)}-x^{(i+1)})^2\right]-x^{(1)}\right\rbrace \nonumber.
\end{align}
 for which from the specific initial condition $x_0 = 0$, no algorithm can perform better. This ``worst-case'' function is defined for signals in a Hilbert space $x \in {l}_2$ and is strongly convex since 
\begin{align}
\mu I \leq f_{Nest}''(x) \leq \mu Q_fI.
\end{align}

From the first order optimality condition
\begin{align}
f'_{Nest}(x^*) = 0,
\end{align}
it follows that any sequence that minimises this function must satisfy
\begin{subequations}\label{sys_opt2}\begin{align}
\frac{Q_f+1}{Q_f-1}(x^{*})^{(1)}-(x^{*})^{(2)} &= 1 \\
(x^{*})^{(k+1)}-2\frac{Q_f+1}{Q_f-1}(x^{*})^{(k)}+x^{*(k-1)} &= 0, \quad k = 2, 3, \dots
\end{align}\end{subequations}
This sequence can be regarded as a linear system with characteristic equation
\begin{align}
q^2-2\frac{Q_f+1}{Q_f-1}q +1 = 0.
\end{align}
The smallest root of this equation is $q = \frac{\sqrt{Q_f}-1}{\sqrt{Q_f}+1}$ \cite[Chapter 2]{nest} and  so the optimal sequence then satisfies
\begin{align}\label{eqn:nest}
(x^*)^{(k)} = q^k, \quad k = 1,2, \dots,
\end{align}
hence
\begin{align}\label{opt_bound}
\|x_k-x^*\|_2^2 \geq \left(\frac{\sqrt{Q_f}-1}{\sqrt{Q_f}+1}\right)^{2k}\|x_0-x^*\|_2^2.
\end{align}
This lower bound for the error $\|x_k-x^*\|$  is the same as the upper bound from the fast gradient method given in \eqref{fg_rate}. Hence, no algorithm can perform better for this  particular function and initial condition, justifying its title of being an optimal method.

To compare the proposed algorithms against such a benchmark, the minimisation of a truncated version of this worst-case function
\begin{align}
&f^n_{Nest}(x) = \frac{1}{2}\mu\|x\|^2 \\
& + \frac{\mu(Q_f-1)}{4}\left\lbrace\frac{1}{2}\left[(x^{(1)})^2 + \sum_{i = 1}^{n} (x^{(i)}-x^{(i+1)})^2\right]-x^{(1)}\right\rbrace \nonumber
\end{align}
was considered with  $n = 10^3$ and $Q_f = 10^6$ starting from the same initial condition $x_0 = 0$. The results of this minimisation are shown in Figure \ref{fig:nest_ce} where the rate $\|x_k-x^*\|^2_2/\|x_0-x^*\|^2_2$ generated by the algorithm $\Sigma_{ml}^6$ and the bound from \eqref{opt_bound} are plotted. Notably, for the truncated function, $\Sigma_{6}^{ml}$ could actually violate the bound. However, \textbf{no claim is made here against the optimality of $\mathcal{FG}$ for $f_{wc}(x)$}. The violation was only achieved by truncating the function and sufficiently increasing the condition number $Q_f$ such that, by a uniqueness of solution argument, the minimiser \eqref{sys_opt2} is corrupted, making the bound in \eqref{opt_bound} rather irrelevant for $f^n_{wc}(x)$. 

\subsection{Non-convex functions}\label{sec:non_convex}
Due to applications in machine learning, recent results on accelerated gradient methods have focussed on functions that are only weakly convex or even non-convex \cite{accel_saddle, perturbed}. For such problems, gradient descent typically gets stuck in local minima/saddle points from which random perturbations in the gradient may have to be added to escape \cite{perturbed}. It has also been noted that introducing momentum (via the fast gradient) can be beneficial for such problems, as the iterates can then overshoot the local minima \cite{accel_saddle}.  

This idea of adding momentum for non-convex problems is extended here to the accelerated $\Sigma^{ml}_N$ method. No rigorous convergence analysis is provided as that goes beyond the scope of the paper. The justification of using $\Sigma_N^{ml}$ for such problems is that the generated points $y_k$ from which the gradient step is taken from may not lie within the convex hull of the iterate history. This allows the algorithm's iterates to leave the local minima. In essence, the randomness of the stochastic gradient method is replaced by the jumps and local instabilities of the deterministic $\Sigma^{ml}_N$ algorithm. 

For these problems, the choice of $\mu$ and $L$ was found to significantly influence the convergence. Making $\mu/L$ too small meant that $\Sigma^{ml}_N$ converged to a local minimum via gradient descent while making $\mu/L$ too large meant that even the gradient descent algorithm became divergent. Best performance was achieved when these parameters were set such that the resulting algorithms were on the boundary between contracting and diverging.  

Two benchmark non-convex functions are examined; Rosenbrock's banana function and the Rastringin function. Again, it is highlighted that a deterministic gradient-based method was used for these minimisations. 



\subsubsection{Rosenbrock banana function}
The Rosenbrock banana function \cite{rosenbrock} with $n = 2$
\begin{align}
f_{Ros}(x) = (1-x^{(1)})^2 + 100(x^{(2)}-x^{(1)})^2
\end{align}
has a global minimiser at $(1,1)$  at the bottom of a valley. For gradient methods, converging to this valley is trivial but then travelling down to the global minimiser is exhaustive. Figure \ref{fig:Ros} shows the minimisation of this function using the algorithms $\mathcal{GD}$, $\mathcal{FG}$ and $\Sigma_N^{ml}$ with $N = 3, \,\dots, \, 9$, $\mu = 10^{-5}$ and $L = 0.9 \times 10^3$ from $x_0 =(-1,1)$. The figure shows the convergence of gradient descent being slow both to and in the valley, and furthermore, for this choice of $\mu$ and $L$, the fast gradient method was divergent. This illustrates the well-known lack of robustness of this algorithm \cite{robust_grad}. The algorithms $\Sigma_N^{ml}$ performed well for this function, with $\Sigma_9^{ml}$ finding a function value of $7.58 \times 10^{-12}$ in 43 iterations (with this number of iterations not accounting for the preliminary steps taken by the algorithm in the restarting). This is lower than the function value of $1.35\times 10^{-10}$ in 185 iterations of the Nelder-Mead method \cite{nelder}. 


\begin{figure}
\graphicspath{ {./Figures/} }
\includegraphics[width=0.5\textwidth]{{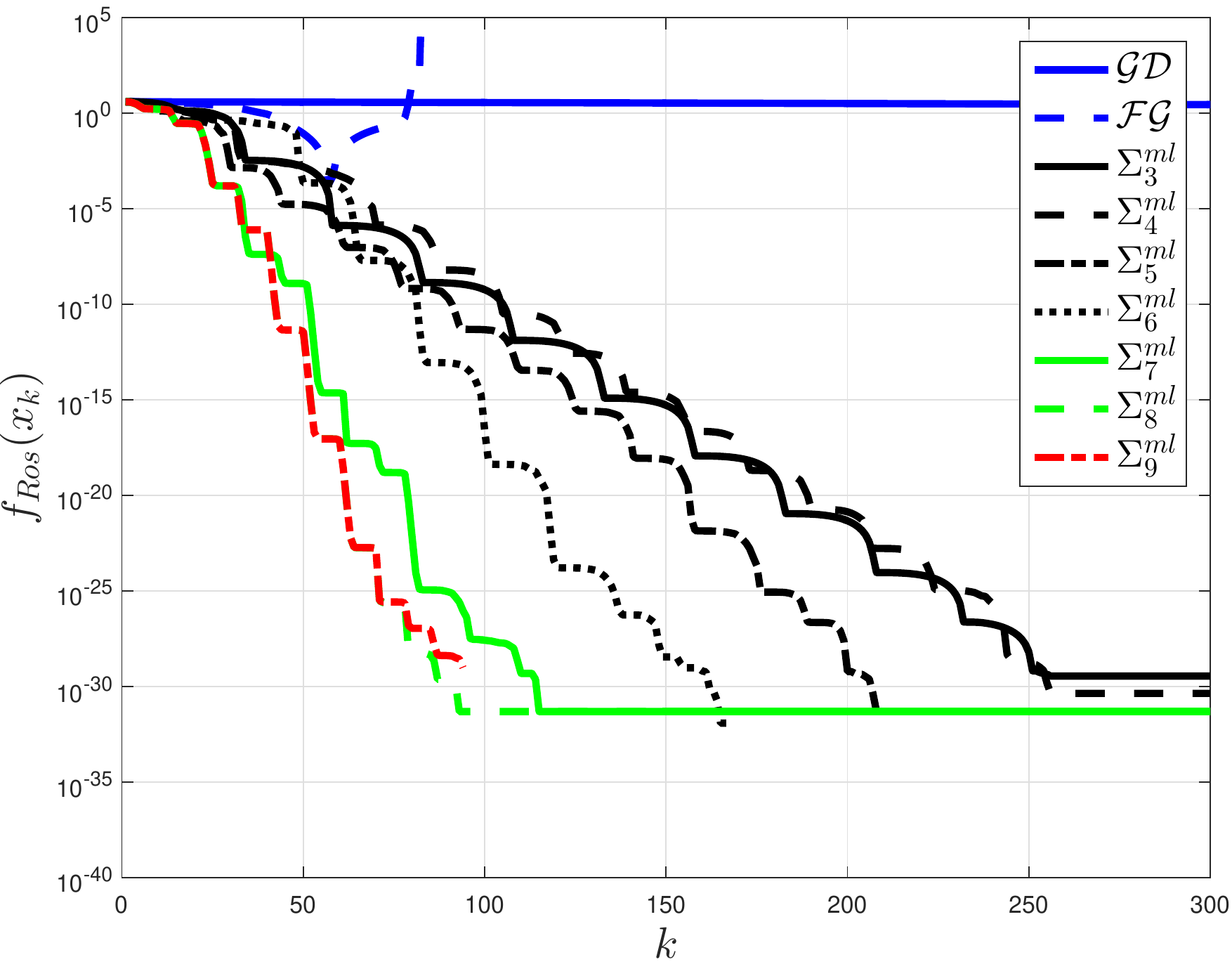}}
\caption{Minimisation of the Rosenbrock banana function using gradient descent, the fast gradient method and the algorithms $\Sigma_N^{ml}$.  }
\label{fig:Ros}
\end{figure}

\subsubsection{Rastrigin function}

A more challenging benchmark non-convex function is the Rastrigin function \cite{rastrigin}
\begin{align}
f_{Rast}(x) = 10n+\sum_{j = 1}^nx_j^2-10\cos(2\pi x_j)
\end{align}
which has a global minima at the origin  and many local minima. The minimisation of this function with $\mu = 1$, $L = 140$ and $n = 2$ is considered here. These parameters were chosen from considering the balance between algorithm convergence and ability to leave local minima, and, strikingly, this choice of $L$ is greater than the curvature at the minimiser. Figure \ref{fig:ras_x0} shows the trace of the iterates for this minimisation using $\Sigma_N^{ml}$ from the initial conditions $x_0 = (5,5)$ and $x_0 = (-5,-3)$. From $x_0 = (5,5)$, a function value of $f_{Rast} =  10^{-6}$ was obtained after 463 iterations. However, the algorithm performed less well from $x_0 = (-5,-3)$, with the obtained functions values fluctuating around $f_{Rast} \approx 5$ after an initial decrease from $f_{Rast} =34.4$. The local minima of the function can be clearly seen in this figure, with the iterates leaving the local dimples and fluctuating near the minimiser. 

Figure \ref{fig:ras2} further evaluates the influence of the initial conditions for this minimisation of this function. The figure plots the log of the minimum value of the cost function obtained by the $\Sigma_{6}^{ml}$ algorithm in 1000 iterations. Typically, for initial conditions on the 8 pointed star apparent in the figure,  function values between $10^{-6}$ and $10^{-8}$ were found. However, everywhere else, function values only between $ 10^{-3}$ and $10^{-1}$ were obtained. In contrast, for this function, gradient descent got trapped in the local minima and the fast gradient method was unstable for this choice of $\mu$ and $L$ which were chosen to escape the local minima. 



\begin{figure}
\centering
\begin{subfigure}[b]{0.5\textwidth}\centering
\includegraphics[width=0.95\textwidth]{{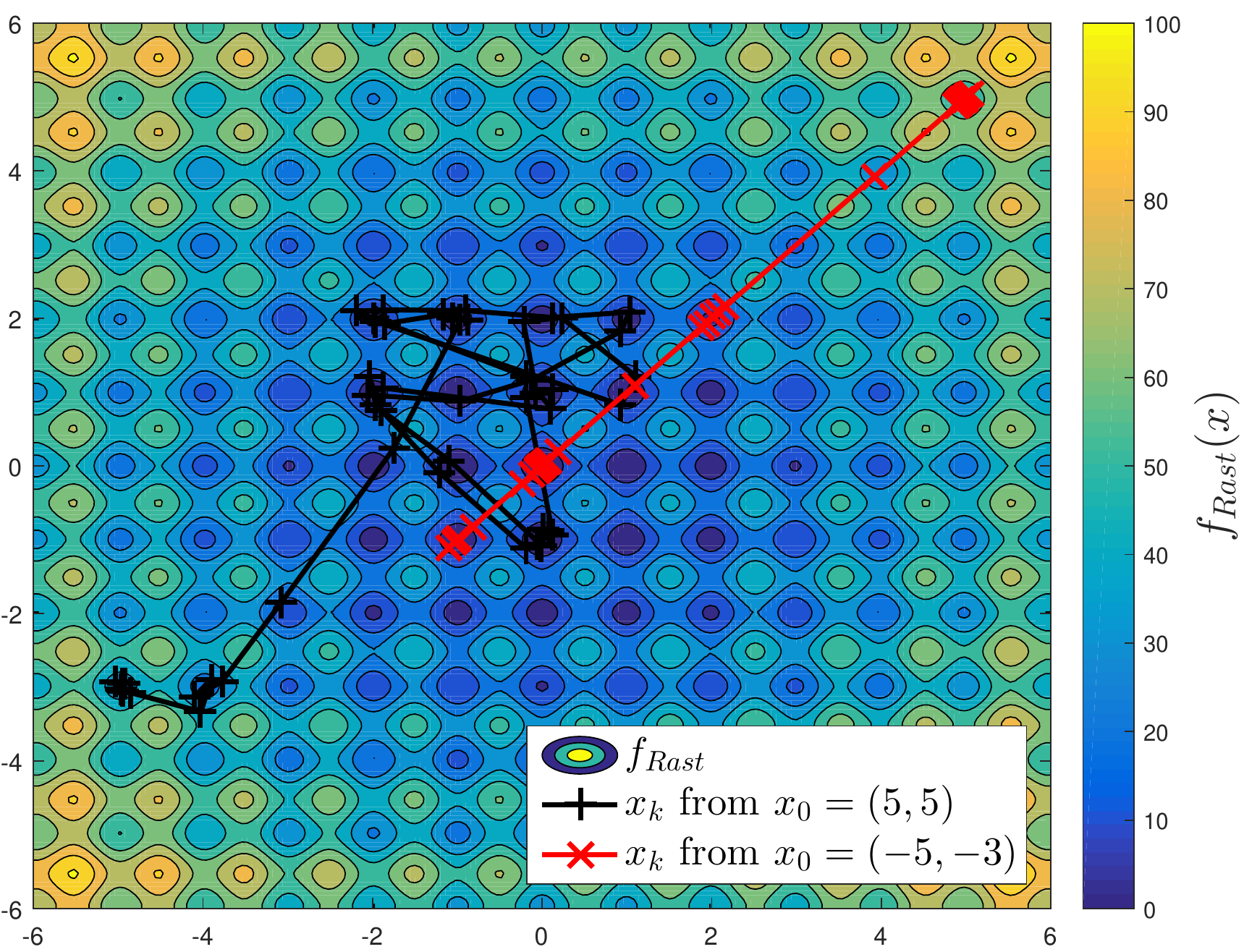}}
\caption{The Rastrigin function $f_{Rast}$ with iterates generated by the algorithm $\Sigma_6^{ml}$ starting from $x_0 = (5,5)$ and $x_0 = (-5,-3)$. The iterates are able to leave the local minima before fluctuating around the global minimiser at the origin.}
\label{fig:ras_x0}
\end{subfigure}
\\
\begin{subfigure}[b]{0.5\textwidth}\centering
\includegraphics[width=0.95\textwidth]{{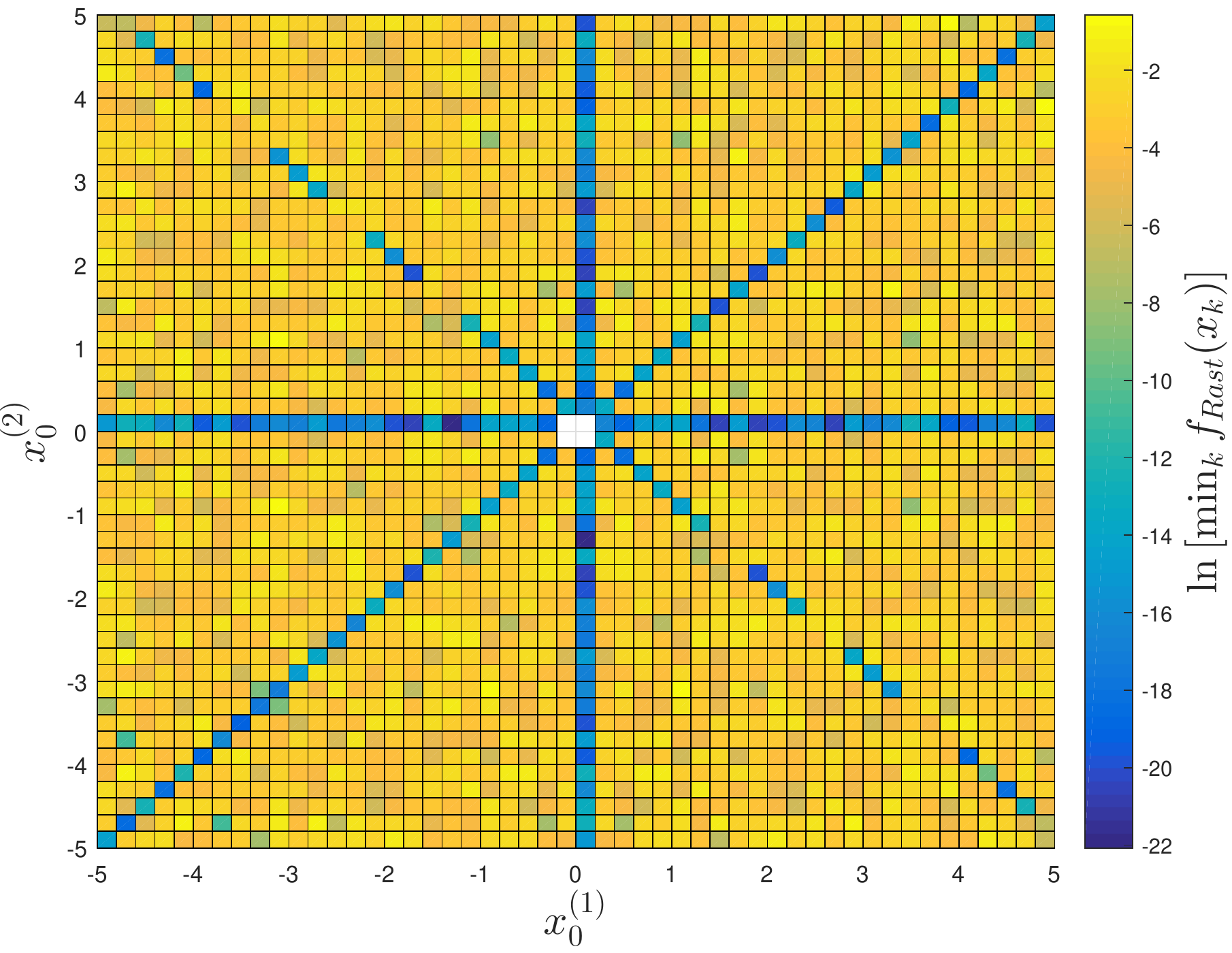}}
\caption{Logarithm of the minimum value of the Rastrigin function $f_{Rast}$ found by $\Sigma_6^{ml}$ from various initial conditions $(x^{1}_0,x^{(2)}_0)$. Initial conditions along the star performed best.}
\label{fig:ras2}
\end{subfigure}
\caption{Figures for the minimisation of the Rastrigin function $f_{Rast}$.}
\label{fig:ras}
\end{figure}

\section{Conclusions}

A set of first-order algorithms  was proposed for minimising strongly convex functions. The proposed algorithms have a similar structure to the fast gradient method but use an increased number of historical iterates, giving them memory. The algorithms are parameterised by considering quadratic functions such that the error associated to the dominant eigenvalue of the Hessian converges at the rate $1-(\mu/L)^\frac{1}{N}$ for generic $N \in \mathcal{N}$. By parametrising against this mode, the algorithms lose robustness, which is recovered by the use of a switching controller. This switching controller uses adaptive restarting such that at each iteration, the algorithms switch between a set of potential iterates to recover monotonic convergence and the acceleration. Numerical examples showcase the benefits of the proposed approach, both for strongly convex and non-convex functions. For example, the algorithms were shown to find values between $10^{-1}$ and $10^{-3}$ of the Rastrigin function from generic initial conditions and values between $10^{-6}$ and $10^{-8}$ with initial conditions along several lines intersecting the global minimiser. Also, when applied to the Rosenbrock banana function, the proposed algorithms could find a function  value of $7.58 \times 10^{-12}$ in 43 iterations from an initial condition of $x_0 = (-1,1)$. It is hoped that the proposed methods will encourage further interest in accelerating first order methods still further using the algorithm memory and to robustify them by closing feedback loops.

\bibliographystyle{IEEEtrans}
\bibliography{bibliog}

\end{document}

\end{document}